\documentclass[12pt,a4paper]{article}
\usepackage{latexsym}
\usepackage{amsfonts}
\usepackage{amssymb}
\usepackage{mathptmx}
\usepackage{amsmath,amsthm}
\newtheorem{theorem}{Theorem}
\newtheorem*{Ch}{Charatheodory's Theorem}
\newtheorem*{FJ}{Fritz John's Theorem}
\newtheorem*{Separation}{Strong Separation Theorem}
\newtheorem{proposition}{Proposition}
\newtheorem{corollary}{Corollary}
\newtheorem{lemma}{Lemma}
\newtheorem{definition}{Definition}
\newtheorem{remark}{Remark}
\newtheorem{example}{Example}

\newcommand{\R}{\mathbf R}

\newcommand{\f}{f : X\to\R^n}
\newcommand{\g}{g : X\to\R^m}

\newcommand{\norm}[1]{\lVert#1\rVert}

\title{Optimality Conditions with  Generalized Convex Vector Problems over Cones} 
\author{Vsevolod I. Ivanov
\\ \small Department of Mathematics, Technical University of Varna, Varna, Bulgaria
\\ \small Email: vsevolodivanov@yahoo.com 
}
\date{}

\begin{document}
\maketitle

\begin{abstract}
The object of investigation in this paper are  generalized convex vector nonlinear programming problems with cone constraints.
We prove that a quasiconvex vector  function is pseudoconvex with respect to the same cone if and only if every vector critical point of the function is a weak global minimizer, a result, which is a natural extension of a known characterization of pseudoconvex scalar functions.

We  introduce a new notion of a Fr\'echet differentiable pseudoconvex problem. We derive that a problem with quasiconvex vector  data is pseudoconvex if and only if every Fritz John vector critical point is a weakly efficient global solution.  Thus, we generalize a lot of previous optimality conditions, concerning the scalar case and the multiobjective Paretian one. 

We also introduce the notion of a Fritz John cone-pseudoinvex  vector problem. We prove that a problem with cone constraints is Fritz John pseudoinvex if and only if every vector critical point of Fritz John type is a weak global minimizer with respect to the considered cone.. 

2010 {\it Mathematics Subject Classification:} 90C26; \quad 90C29; \quad 90C46; \quad 26B12; \quad 26B25

{\it Key words and phrases:} vector optimization; cone optimization; Fritz John  optimality conditions;  pseudoconvex and quasiconvex vector  functions;  FJ-pseudoconvex and FJ-pseudoinvex problems with respect to some cone

\end{abstract}
\section{Introduction}
Pseudoconvex and quasiconvex functions play important role in scalar optimization. The necessary conditions for optimality become sufficient when the objective function is pseudoconvex and the inequality constraints are quasiconvex. They can be generalized to various classes of vector  functions. Some of these  extensions have similar role as in the scalar case. Several classes of quasiconvex and pseudoconvex vector functions have been studied in optimization theory (see the papers \cite{cam96,cc00} and the books \cite{luc89,luc05,cam05,jah11} and the reference therein).  In some of them,  the authors consider generalized convex vector  functions, which satisfy other generalized convexity assumptions with known sufficient optimality conditions. Really, the most of these notions were not applied in sufficient optimality conditions. There are several works, where were obtained optimality conditions in vector problems with pseudoconvex and quasiconvex vector  data (\cite{jahn86,cam98}). 
These results can be seen in the books \cite{jah11}, where the cone constrained case is considered, and \cite{cam05,ggt04}, where the problem with inequality constraints is studied.
Optimality conditions for a  problem with scalar objective function and quasiconvex vector  constraint function were derived in \cite{jey93}.

In this paper, we  introduce a new more general vector problem with a cone constraint, whose objective function is pseudoconvex and the constraint is quasiconvex. We generalize some known sufficient optimality conditions to problems with more general data. We prove that a problem with quasiconvex objective function and strictly scalarly quasiconvex constraint fuction is pseudoconvex if and only if every Fritz John critical point is a weakly efficient global solution. 

In 1985 Martin \cite{mar85} defined Kuhn-Tucher invex scalar problems. He proved that these problems are the most general ones with the property that every Kuhn-Tucker stationary point is a global minimizer. Later this result were generalized to vector problems with inequality constraints, where the Paretian cone were applied (see Chapter 1 in the book \cite{ara2010} and the references therein).  
Several papers appeared, which concerned optimality conditions in multiobjective optimization with Kuhn-Tucker pseudoinvex data and Fritz John ones, but the authors considered the Paretian case in all of them \cite{osu98}. Nobody extended these results to the more complicated case with arbitrary cones.  The locally Lipscitz case was threeted in the paper \cite{sac05}.

We define the widest class of Fr\'echet differentiable cone  vector problems where the minimization is with respect to another cone and such that every Fritz John stationary point is a weakly efficient global solution. We prove that no further generalizations are possible. We  find the maximal extension of the problem with a cone constraint such that the necessary conditions are sufficient. We call such problem Fritz John pseudoinvex (FJ-pseudoinvex). We prove that a problem with  cone constraints is Fritz John pseudoinvex with respect to some cones $C$ and $K$ if and only if every vector critical point of Fritz John type is a global minimizer.

The paper is organized as follows: In section \ref{s2} we prove that a quasiconvex vector  function with respect to some cone $C$ is pseudoconvex with respect to the same cone if and only if every vector critical point of Fritz John type is a weak global minimizer. This result is a generalization to vector  functions of a known claim due to Crouzeix and Ferland \cite{cro82}. A generalization of this result for nonsmooth functions were obtained by Komlosi \cite{kom83}.
In section \ref{s3}, we obtain sufficient optimality conditions for a new problem with pseudoconvex vector  objective function and quasiconvex cone constraints. We prove that a problem with a quasiconvex objective function and strictly scalarly quasiconvex constraint is Fritz John pseudoconvex if and only if each vector critical point of Fritz John type is a weakly efficient global solution. This claim is a generalization of the respective result due to Ivanov \cite{JOGO-1} concerning the scalar problem with inequality constraints.
In section \ref{s4}, we extend the sufficient optimality conditions to problems with pseudoinvex data. We prove that this class is the widest one such that every vector critical point of Fritz John type is a weakly efficient global solution.

\section{Preliminaries}
Let $C\in\R^n$ and $K\in\R^m$  be  closed convex cones with nonempty interiors ${\rm int}(C)$ and ${\rm int}(K)$, whose verteces are the origins $0$ of the respective space. Denote by $a\cdot b$ the scalar product between the vectors $a\in\R^n$ and $b\in\R^n$, by $J f(x)$ the Jacobian matrix of a Fr\'echet differentiable vector-valued function $f$ at some point $x$.  Suppose that $\f$ is a given Fr\'echet differentiable function, defined on an open set $X\subset\R^s$. Denote the positive polar cone of $C$ by $C^*$, that is 
\[
C^*:=\{\lambda\in R^n\mid\lambda\cdot x\ge 0\textrm{ for all }x\in C\},
\]
 the positive polar cone of $C^*$ by $C^{**}$,  and  by $C\times K$ the Cartesian product of the cones $C$ and $K$.

\begin{definition}[\cite{ggt04}]
A cone $P$ with a vertex at the origin is called pointed iff $P\cap (-P)=\{0\}$.
\end{definition}

We continue with some preliminary results.

\begin{Ch}[\cite{ggt04}]
Let $A$ be a nonempty set in the $n$-dimensional space $\R^n$. Then every point from the convex hull ${\rm conv}(A)$ is a convex combination of $n+1$ or less points from $A$.
\end{Ch}

\begin{Separation}[\cite{ggt04}]
Let $X$ be a nonempty closed convex set in $\R^n$ and $y\notin X$. Then there exists a nonzero vector $\alpha$ in $\R^n$ such that
\[
\sup_{x\in X}(\alpha\cdot x)<\alpha\cdot y.
\]
\end{Separation}

\begin{proposition}[\cite{ggt04}]\label{C**}
Let $C$ be a nonempty closed convex cone in the $n$-dimensional space $\R^n$, whose vertex is the origin. Then $C^{**}=C$.
\end{proposition}

The following lemma directly follows from  Strong separation theorem (\cite{ggt04}), but we provide another proof for completeness, which is very short.

\begin{lemma}\label{lema1}
Let $C$ be a closed convex cone, and $x\notin C$. Then there exists $\lambda\in C^*$ such that $\lambda\cdot x<0$.
\end{lemma}
\begin{proof}
Assume the contrary that $\lambda\cdot x\ge 0$ for all $\lambda\in C^*$. It follows from here that $x\in C^{**}$. On the other hand, by Proposition \ref{C**}, we have $C=C^{**}$, which contradicts the hypothesis $x\notin C$.
\end{proof}

\begin{lemma}\label{lema2}
Let $C\subset\R^n$ be a cone and $\lambda\in C^*$, $\lambda\ne 0$. Then $\lambda\cdot x>0$ for all $x\in{\rm int}(C)$ such that $x\ne 0$.
\end{lemma}
\begin{proof}
Suppose the contrary that there exists $x\in {\rm int}(C)$ with $\lambda\cdot x\le 0$, $x\ne 0$. It follows from the definition of a positive polar cone that $\lambda\cdot x=0$. There exists a number $\delta>0$ such that $x-\delta\lambda\in{\rm int}(C)$, because $\lambda\in\R^n$. By $\lambda\in C^*$ we have
$\lambda\cdot(x-\delta\lambda)\ge 0$. We obtain from here that $\lambda=0$, which is a contradiction.
\end{proof}

\begin{lemma}\label{lema3}
Let $C$ be a closed convex cone and $x\in C$. Then  $x\in {\rm int}(C)$ if and only if $\lambda\cdot x>0$ for all $\lambda\in C^*$ with $\lambda\ne 0$.
\end{lemma}
\begin{proof} 
Suppose that $\lambda\cdot x>0$ for all $\lambda\in C^*$ with $\lambda\ne 0$, but $x\notin {\rm int}(C)$.  It follows from $x\notin{\rm int}(C)$ that there exists an infinite sequence $x_k$, converging to $x$, such that $x_k\notin C$. It follows from Lemma \ref{lema1} that there exists $\lambda_k\in C^*$ such that $\lambda_k\cdot x_k<0$. We conclude from here that $\lambda_k\ne 0$. Without loss of generality we suppose that
$\norm{\lambda_k}=1$ for all positive integers $k$. Passing to a subsequence we could suppose that $\lambda_k$ converges to some point $\lambda_0\ne 0$. Taking the limits when $k\to +\infty$ we obtain that $\lambda_0\cdot x\le 0$. Since the polar cone is always closed, we conclude that $\lambda_0\in C^*$. We conclude from here that $\lambda_0\cdot x>0$, which is a contradiction.

The converse part of the proof follows from Lemma \ref{lema2}. 
\end{proof}

Lemma \ref{lema3} in a different form appears in the book \cite{jah11} (see  Lemma 3.21 (b)).

\begin{lemma}\label{lema7}
Let $C$ be a closed convex cone 
 and $\lambda\in C^*$. Then  $\lambda\in {\rm int}(C^*)$ if and only if $\lambda\cdot x>0$ for all $x\in C$, $x\ne 0$.
\end{lemma}
\begin{proof}
Let $\lambda\in {\rm int}(C^*)$. According to Lemma \ref{lema3} we have $\lambda\cdot x>0$ for all $x\in C^{**}$, $x\ne 0$. Then the claim follows from the relation $C^{**}=C$, because $C$ is closed and convex.

Conversely, suppose that $\lambda\cdot x>0$ for all $x\in C$, $x\ne 0$. Then, by $C=C^{**}$ we have $\lambda\cdot x>0$ for all $x\in C^{**}$,  $x\ne 0$. Thus, the claim follows again from Lemma \ref{lema3}. 
\end{proof}

\begin{lemma}\label{lema4}
Let $C\subset\R^n$ be an arbitrary closed convex cone with a vertex at the origin. Then $0\notin{\rm int}(C^*)$ if and only if $C\ne\{0\}$.
\end{lemma}
\begin{proof}
Suppose that $C\ne\{0\}$. Assume the contrary that $0\in{\rm int}(C^*)$. Therefore, there exists $\delta>0$ such that $\norm{\lambda}<\delta$ implies that $\lambda\in C^*$. Let us take an arbitrary point $x\in\R^n$, $x\ne 0$. Then 
$(\delta/2)\, x/\norm{x}$ belongs to a neigbourhood of the origin with a radius $\delta$. Since $C^*$ is a cone, then we conclude that $x\in C^*$. 
Therefore $C^*$ coincides with the whole space. Hence, $C=C^{**}=\{0\}$, which is a contradiction.

Suppose that $0\notin{\rm int}(C^*)$, but $C=\{0\}$. Therefore, $C^*\equiv\R^n$ and $0\in{\rm int}(C^*)$, a contradiction.
\end{proof}

\begin{corollary}\label{lema8}
Let $C\subset\R^n$ be an arbitrary closed convex cone with a vertex at the origin such that $C\ne\{0\}$ and ${\rm int}(C^*)\ne\emptyset$ . Let $\lambda\in{\rm int}(C^*)$. Then $\lambda\ne 0$.
\end{corollary}

\begin{lemma}\label{lema5}
Let $C\subset\R^n$ and $K\subset\R^m$ be arbitrary closed convex cones  with verteces at the origin of the respective space. Suppose that the positive polars of the cones $C$ and $K$ have nonempty interior. 
Then 
\[
{\rm int}(C^*)\times {\rm int}(K^*)\subset {\rm int} (C\times K)^*.
\]
\end{lemma}
\begin{proof}
Let $C\ne\{0\}$ and $K\ne\{0\}$. Suppose that $(\lambda,\mu)\in C^*\times K^*$ is an arbitrary point such that $\lambda\in{\rm int}(C^*)$ and $\mu\in{\rm int}(K^*)$. Therefore, by Corollary \ref{lema8}, $\lambda\ne 0$ and $\mu\ne 0$.  Choose any point $(a,b)\in C\times K$, $a\in C$, $b\in K$, $(a,b)\ne (0,0)$. Then at least one of the numbers $a$ and $b$ is different from zero. Without loss of generality we could suppose that $a\ne 0$. Then it follows from Lemma \ref{lema7} that  $\lambda\cdot a>0$. On the other hand we have $\mu\cdot b\ge 0$.  Therefore, $\lambda\cdot a+\mu\cdot b> 0$.  Since $a$ and $b$ are arbitrary points from $C$ and $K$, respectively, then by Lemma \ref{lema7}, we obtain that $(\lambda,\mu)\in {\rm int}(C\times K)^*$.

It is easy to see that the claim also holds if $C\ne\{0\}$ and $K=\{0\}$ repeating approximately the same proof. In this case, $b=0$, $\mu\in\R^m$ is an arbitrary point and $\lambda\cdot a+\mu\cdot b>0$. Similarly, the claim also is satisfied if $C=\{0\}$ and $K\ne\{0\}$. The case $C=\{0\}$ and $K=\{0\}$ is trivial.
\end{proof}

The next result is a type of separation theorem.

\begin{lemma}\label{lema6}
Let $K\in\R^m$ be a closed convex pointed cone with a vertex at the origin $0$. Then there exists a vector $\eta\in\R^m$, $\eta\ne 0$ such that
\[
\eta\cdot x<0,\quad\forall x\in K,\; x\ne 0.
\]
\end{lemma}
\begin{proof}
Denote by $\tilde K$ the set $\tilde K:=\{x\in K: \norm{x}=1\}$. Let $B$ be the convex hull ${\rm conv}(\tilde K)$ of the set $\tilde K$. We prove that the origin $0$ does not belong to the set $B$, that is $0\notin B$. Suppose the contrary that $0\in B$. By Charatheodory's theorem $0$ is a convex combination of $n+1$ or less points from $\tilde K$. Therefore, there exists a positive integer $k$ with $1\le k\le n+1$, points $x_1$, $x_2,\dots$, $x_k\in\tilde K$ and nonnegative numbers $\alpha_1$, $\alpha_2,\dots$, $\alpha_k$ with
$\sum_{i=1}^k\alpha_i=1$ such that
\[
0=\sum_{i=1}^k\alpha_i x_i.
\]
At least one of the numbers $\alpha_i$ is strictly positive. Without loss of generality, we can suppose that $\alpha_1>0$.
The case $k=1$ is impossible because $\norm{x_1}=1$.
Let $y=\alpha_1 x_1$ and $z=\sum_{i=2}^k\alpha_i x_i$. We have $y\in K$, $z\in K$ and $y+z=0$, because $K$ is a convex cone. Since $K$ is a pointed cone, then $y=0$, which contradicts the assumption $\norm{x_1}=1$, because $\alpha_1>0$. 
Thus $0\notin B$. Then, it follows from strong separation theorem that there exists a vector $\eta\in\R^m$, $\eta\ne 0$ such that
\[
\sup_{x\in B}\, \eta\cdot x<\eta\cdot 0=0.
\]
Therefore $\eta\cdot x<0$ for all $x\in K\setminus\{0\}$.
\end{proof}

\section{Minimization of pseudoconvex and quasiconvex vector  functions over cones}
\label{s2}
Several definitions for pseudoconvex  and quasiconvex vector-valued functions are introduced in optimization (\cite{luc90,cam96,ben02,tor10}). There are no clear relations between some of these classes of functions. 
The necessary conditions for optimality of a pseudoconvex scalar function are also sufficient. This property is not satisfied for quasiconvex functions. For scalar functions, every pseudoconvex  Fr\'echet differentiable function is quasiconvex. The converse claim holds under additional assumptions. Crouzeix and Ferland have proved \cite{cro82} that a quasiconvex Fr\'echet differentiable function is pseudoconvex if and only if every stationary point is a global minimizer. 

In this section, we consider one definition for pseudoconvexity and one for quasiconvexity. We also extend the theorem by Crouzeix and Ferland to vector-valued functions.
We begin the section with some preliminary definitions.

\begin{definition}[\cite{ggt04}]
Let $f:X\to\R$ be a scalar function, defined on some convex set $X\in\R^s$. Then it is called
\begin{itemize}
\item
quasiconvex iff $f[x+t(y-x)]\le\max\{f(x),f(y)\}$ for all $x$, $y\in X$ and each $t\in[0,1]$ 
\item
strictly quasiconvex iff $f[x+t(y-x)]<\max\{f(x),f(y)\}$ for all $x$, $y\in X$ such that $y\ne x$ and each $t\in(0,1)$.
\end{itemize}
\end{definition}

\begin{definition}[\cite{ggt04}]
A Fr\'echet differentiable scalar function $f:X\to\R$ is called pseudoconvex on an open set $X\in\R^s$ iff the following implication is satisfied for all $x$, $y\in X$:
\[
f(y)<f(x)\quad\Rightarrow\quad\nabla f(x)(y-x)<0
\]
\end{definition}

Quasiconvexity of scalar functions over open convex sets is equivalent to the following implication \cite{ggt04}:
\[
f(y)\le f(x),\; x\in X,\; y\in X\quad\Rightarrow\quad\nabla f(x)(y-x)\le 0
\]

These definitions are extended to vector functions.

\begin{definition}[\cite{cam05}]
A Fr\'echet differentiable vector-valued function $\f$ is called quasiconvex on the set $X\in\R^s$ with respect to the cone $C\in\R^n$ iff the following implication is satisfied for all $x$, $y\in X$:
\[
f(y)\in f(x)-{\rm int}(C)\quad\Rightarrow\quad J f(x)(y-x)\in - C.
\]
\end{definition}

\begin{definition}[\cite{cam05}]
A Fr\'echet differentiable vector-valued  function $\f$ is called pseudoconvex on the set $X\subset\R^s$ with respect to the cone $C\in\R^n$ iff the following implication is satisfied for all $x$, $y\in X$:
\[
f(y)\in f(x)-{\rm int}(C)\quad\Rightarrow\quad J f(x)(y-x)\in -{\rm int}(C).
\]
\end{definition}

It is obvious that every pseudoconvex vector function is quasiconvex. The converse does not hold.


\begin{definition}
A point $x$ is called a weak global minimizer of $\f$ with respect to some cone $C\subset\R^n$ (or weakly efficient or weakly effective solution of the minimization problem) iff there is no $y\in C$ such that
$f(y)\in f(x)-{\rm int}(C)$.
\end{definition}

\begin{definition}
Let $\f$ be a given Fr\'echet differentiable vector function. Then a point $x\in X$ is called critical for the function $\f$ with respect to some cone $C\subset\R^n$ iff there exists $\lambda\in C^*$ such that
\[
\lambda\cdot J f(x)=0,\;\lambda\ne 0.
\]
\end{definition}

It is well known that every weak local or global minimizer is a vector critical point. The converse is not always satified. In the next theorem we connect these notions with pseudoconvexity and quasiconvexity of vector functions.
It is a generalization to vector  functions of the condition for pseudoconvexity of a given quasiconvex function from the paper \cite{cro82}:

\begin{theorem}\label{th2}
Let $X\subset\R^n$ be a convex set. Suppose that $C$ is a closed convex cone and $f$ is a Fr\'echet differentiable quasiconvex vector-valued function, defined on $X$. Then $f$ is pseudoconvex  with respect to the cone $C$ if and only if every vector critical  point $x$ is a weak global minimizer of $f$  with respect to the cone $C$.
\end{theorem}
\begin{proof}
Let $f$ be pseudoconvex. 
We prove that for every vector critical point $x$  is a weak global minimizer of $f$. Assume the contrary that $x$ is not a weak minimizer. Then there exists a point $y\in X$ such that $f(y)\in f(x)-{\rm int}(C)$. According to pseudoconvexity of $f$ we have that $J f(x)(y-x)\in -{\rm int}(C)$. Since $x$ is critical, then there exists $\lambda\in C^*\setminus\{0\}$ such that $\lambda\cdot J f(x)=0$. On the other hand,
it follows from $\lambda\in C^*\setminus\{0\}$,  by Lemma \ref{lema2}, that
$
\lambda\cdot J f(x)(y-x)<0,
$
which is a contradiction.

We prove the converse claim. Suppose that every vector critical  point is a weak global minimizer of $f$. We prove that $f$ is pseudoconvex. Let $x\in X$ and $y\in X$ be such that
\[
f(y)\in f(x)-{\rm int}(C).
\]
 It follows from here 
that $x$ is not a weak minimizer. Therefore, 
\[
\lambda\cdot J f(x)\ne 0\quad\textrm{for every}\quad \lambda\in C^*\setminus\{0\}.
\] 
It follows from $f$ is quasiconvex that $J f(x)(y-x)\in - C$. 
 We want to prove that 
 \[
J f(x)(y-x)\in -{\rm int}(C).
\]
 Assume the contrary, that is 
\[
J f(x)(y-x)\in -C\setminus {\rm int}(C).
\]
  Then by Lemma \ref{lema3} we obtain that there exists a multiplier 
\[
\lambda\in C^*\setminus\{0\}\quad\textrm{with}\quad \lambda\cdot J f(x)(y-x)=0.
\]
 Since $f$ is continuous, it follows from $f(y)\in f(x)-{\rm int}(C)$ that there exists $\delta>0$ with
\[
f(y+\delta\lambda\cdot Jf(x))\in f(x)-{\rm int}(C).
\]
 By quasiconvexity of $f$, we obtain that 
\[
Jf(x)(y+\delta\lambda\cdot Jf(x)-x)\in -C.
\]
 Using that $\lambda\in C^*$, we conclude that 
\[
\lambda\cdot Jf(x)(y-x)+\delta\norm{\lambda\cdot Jf(x)}^2\le 0.
\]
 Hence, by $\lambda\cdot Jf(x)(y-x)=0$, we have $\lambda\cdot Jf(x)=0$,  which contradicts our conclusion.
\end{proof}

\section{Cone-pseudoconvex vector problems}
\label{s3}

Consider the multiobjective nonlinear programming problem

\bigskip
C-minimize $f(x)$ subject to $g(x)\in -K$,\hfill (P)
\bigskip

\noindent
where $\f$ and $\g$ are given differentiable vector  functions defined on some open set $X\subset\R^s$, $C$ and $K$ are given closed convex cones with a verteces at the origin. Denote by $S$ the feasible set, that is \[
S:=\{x\in X\mid g(x)\in -K\}.
\]

\begin{definition}
A feasible point $x$ is called weak global minimizer iff there is no another feasible point $y$ such that $f(y)\in f(x)-{\rm int}(C)$.
\end{definition}

The following necessary conditions for optimality of Fritz John type are known (for example, see \cite{jah11}):

\begin{FJ}
Let $x^0$ be a weakly effective solution of the problem {\rm (P)} and the cones $C$ and $K$ have nonempty interior. Moreover, suppose that $f$ and $g$ are Fr\'echet differentiable at $x^0$. Then there exist vectors $\lambda^0\in C^*$ and $\mu^0\in K^*$ such that $(\lambda^0,\mu^0)\ne (0,0)$ and
\begin{equation}\label{1}
\lambda^0\cdot J f(x^0)+\mu^0\cdot J g(x^0)=0,\quad \mu^0\cdot g(x^0)=0.
\end{equation}
\end{FJ}

\begin{definition}
Every point $x^0$, which satisfies these necessary conditions is called a Fritz John stationary (or critical) point.
\end{definition}

In Ref. \cite{cam05} are given assumptions, which ensure that Conditions (\ref{1}) are sufficient for a given point to be a weakly efficient solution of (P). It is supposed in them that the objective function is pseudoconvex and the constraint function is quasiconvex in some sense. We define another notion of pseudoconvexity-quasiconvexity and derive more general sufficient conditions.

\begin{definition}
A function $\g$ is called scalarly quasiconvex on the convex set $X\in\R^s$ with respect to the cone $K$ iff the scalar function $\mu\cdot g(x)$ is a quasiconvex  scalar function of its argument $x$ for every $\mu\in K^*$.
\end{definition}

The scalar quasiconvexity reduce to the condition that all component are quasiconvex functions in the case when $K$ is the Paretian cone 
\[
\R^m_+=\{x=(x_1,x_2,\dots,x_m)\in\R^m\mid x_i\ge 0,\quad i=1,2,\dots, m\}.
\]

For every feasible point $x$ for the problem (P) denote by $M^*(x)$ the cone
\[
M^*(x):=\{\mu\in K^*\mid \mu\cdot g(x)=0,\}
\]
where $K^*$ is the positive polar cone of $K$. Suppose that $M^{**}(x)$ is the positive polar cone of $M^*(x)$.

We consider also the class of constraint functions, which satisfy the following implication:

\medskip
$g(x)\in -K,\; g(y)\in -K\quad$ imply $\quad J g(x)(y-x)\in -M^{**}(x)$.\hfill (QC)
\medskip

\begin{proposition} 
Let the function $\g$ be scalarly quasiconvex on the convex set $X\in\R^s$ with respect to the cone $K$. Then it satisfies implication {\rm (QC)}.
\end{proposition}
\begin{proof}
Let $x$ and $y$ be arbitrary feasible points. By scalar quasiconvexity of $g$ we obtain that the scalar function $\mu\cdot g$ is quasiconvex for every $\mu\in K^*$. Because of $g(y)\in -K$ and $\mu\in M^*(x)$, we obtain that 
\[
\mu\cdot g(y)\le 0=\mu \cdot g(x).
\]
Hence,
\[
\mu\cdot g(x+t(y-x))\le\mu\cdot g(x)\quad\textrm{for all}\quad t\in[0,1]
\]
and
\[
\mu\cdot J g(x)(y-x)\le 0\quad\textrm{for every}\quad \mu\in M^*(x).
\]
 Therefore $Jg(x)(y-x)\in -M^{**}(x)$.
\end{proof}

We introduce the following definition, which is an extension of the notion Fritz John pseudoconvex scalar problems due to Ivanov \cite{JOGO-1}.

\begin{definition} 
We call the problem {\rm (P)} with  Fr\'echet differentiable data  Fritz John pseudoconvex (in short, FJ-pseudoconvex) iff  for all points $x\in X$  and $y\in X$ is satisfied the following implication:
\begin{equation}\label{2}
\left.
\begin{array}{l}
f(y)\in f(x)-{\rm int}(C) \\
g(x)\in -K,\; g(y)\in -K
\end{array}\right]
\quad\Rightarrow \quad
\left[
\begin{array}{l}
J f(x)(y-x)\in -{\rm int}(C^{**})\\
J g(x)(y-x)\in -{\rm int}(M^{**}(x)).
\end{array}
\right.
\end{equation}
\end{definition}

\begin{remark}
In the case when $K$ is the positive orthant in the space $\R^m$, then $\mu\cdot g(x)=0$ implies that $\mu_j g_j(x)=0$ for all $j=1,2,\dots ,m$. Therefore, $\mu_j=0$ for all constraints $j$, which are not active. Hence
\[
J g(x)(y-x)\in -{\rm int}(M^{**}(x))\quad\textrm{implies that}\quad \nabla g_j(x)(y-x)<0
\]
 for all active constraints $j$. If $C=\R^n_+$, then
\[
Jf(x)(y-x)\in -{\rm int}(C^{**})\quad\textrm{implies that}\quad \nabla f_i(x)(y-x)<0\quad\textrm{for all}\quad i=1,2,\dots,n.
\]
\end{remark}

\begin{theorem}\label{th3}
Let the problem {\rm (P)} be  Fr\'echet differentiable and FJ-pseudoconvex. Suppose that $x$ is
a Fritz John stationary point. Then $x$ is a weakly efficient global solution.
\end{theorem}
\begin{proof}
It follows from the hypothesis that there exist $\lambda\in C^*$ and $\mu\in K^*$ such that the following conditions are satisfied
\[
\lambda\cdot J f(x)+\mu\cdot J g(x)=0,\;(\lambda,\mu)\ne (0,0),\;\mu\cdot g(x)=0.
\]
Assume the contrary that $x$ is not weakly efficient. Then there exists a feasible point $y$ such that 
\[
f(y)\in f(x)-{\rm int}(C).
\]
 According to FJ-pseudoconvexity of (P) we have that
\[
J f(x)(y-x)\in -{\rm int} (C^{**})\quad\textrm{and}\quad J g(x)(y-x)\in -{\rm int} (M^{**}(x)).
\] 
It follows from Lemma \ref{lema5} that
\[
{\rm int} (C^{**})\times {\rm int} (M^{**}(x))\subset {\rm int} (C^*\times M^*(x))^*.
\]
Therefore, 
\[
[Jf(x)(y-x),Jg(x)(y-x)]\in -{\rm int} (C^*\times M^*(x))^*.
\]
 It follows from here, from Lemma \ref{lema7} and by
$(\lambda,\mu)\in C^*\times M^*(x)$ that
\[
\lambda\cdot J f(x)(y-x)+\mu\cdot J g(x)(y-x)<0,
\]
which is a contradiction.
\end{proof}

\begin{definition}
According to the definition of strict scalar quasiconvexity we call the constraint function $g$ strictly scalarly quasiconvex at the point $x$ with $g(x)\in -K$ iff the scalar function $\mu\cdot g$ is strictly quasiconvex at the point $x$ for all $\mu\in M^*(x)$, that is
\[
g(y)\in -K,\; y\ne x\quad\Rightarrow\quad \mu\cdot g(z)<\mu\cdot g(x),\;\forall z\in(x,y).
\]
If this inequality is satisfied for every feasible point $x$, then we call $g$ strictly scalarly quasiconvex.
\end{definition}

\begin{theorem}\label{th4}
Let $X$ be a convex set and {\rm (P)} be Fr\'echet differentiable. Suppose that $f$ is quasiconvex and $g$ is strictly scalarly quasiconvex. Then the problem {\rm (P)} is FJ-pseudoconvex if and only if every Fritz John stationary point $x$ 
is a  weakly efficient global solution of {\rm (P)}.
\end{theorem}
\begin{proof}
Let (P) be FJ-pseudoconvex. Then it follows from Theorem \ref{th3} that  every vector critical Fritz John  point $x$ is a  weakly efficient global solution of {\rm (P)}.

Conversely, suppose that every vector critical Fritz John  point $x$ is a  weakly efficient global solution of {\rm (P)}. We prove that (P) is FJ-pseudoconvex.
Let $x$ and $y$ be any feasible points for (P) such that $f(y)\in f(x)-{\rm int}(C)$.
Therefore, $x$ is not weakly efficient. We conclude from here that $x$ is not a Fritz John stationary point. Therefore,
\begin{equation}\label{10}
\lambda\cdot J f(x)+\mu\cdot J g(x)\ne 0
\end{equation}
for all points $\lambda\in C^*$, $\mu\in K^*$ such that $(\lambda,\mu)\ne (0,0)$, $\mu\cdot g(x)=0$. 

We prove that 
\[
J f(x)(y-x)\in -{\rm int}(C).
\]
  It follows from quasiconvexity of $f$ that $J f(x)(y-x)\in - C$. Suppose that
\[
J f(x)(y-x)\in -C\setminus {\rm int}(C).
\]
 It follows from Lemma \ref{lema3} that there exists 
\[
\lambda\in C^*,\;\; \lambda\ne 0\quad\textrm{such that}\quad \lambda\cdot Jf(x)(y-x)=0.
\]
 Since $f$ is continuous, it follows from $f(y)\in f(x)-{\rm int}(C)$ that there exists $\delta>0$ with 
\[
f(y+\delta\lambda\cdot Jf(x))\in f(x)-{\rm int}(C).
\]
 By quasiconvexity of $f$, we obtain that
\[
Jf(x)(y+\delta\,\lambda\cdot Jf(x)-x)\in -C.
\]
 Using that $\lambda\in C^*$, we conclude that
 \[
\lambda\cdot Jf(x)(y-x)+\delta\norm{\lambda\cdot Jf(x)}^2\le 0.
\]
 Hence, by $\lambda\cdot Jf(x)(y-x)=0$, we have $\lambda\cdot Jf(x)=0$. Let us choose in (\ref{10}) $\mu=0$ and keep the same point $\lambda$. Thus, we have a contradiction to inequality (\ref{10}). 

We prove that 
\[
J g(x)(y-x)\in -{\rm int}(M^{**}(x)).
\]
By the strict scalar quasiconvexity of $g$ we obtain that the scalar function $\mu\cdot g$ is strictly quasiconvex for every $\mu\in M^*(x)$. 
Hence, 
\[
\mu\cdot g(x+t(y-x))<\mu\cdot g(x)=0\quad\textrm{for all}\quad t\in(0,1),
\]
 because $y\ne x$. Therefore $\mu\cdot J g(x)(y-x)\le 0$ for every $\mu\in M^*(x)$ and $Jg(x)(y-x)\in -M^{**}(x)$. We prove that
\[
Jg(x)(y-x)\in -{\rm int}(M^{**}(x)).
\]
 Suppose the contrary that
\[
Jg(x)(y-x)\in -M^{**}(x)\setminus {\rm int}(M^{**}(x)).
\]
Then it follows from Lemma \ref{lema7} that there exists
\[
\alpha\in M^{*}(x),\;\; \alpha\ne 0\quad\textrm{such that}\quad \alpha\cdot Jg(x)(y-x)=0.
\]
 Let us take an arbitrary point $z$ from the segment $(x,y)$. It follows from the continuity of $\alpha\cdot g$ that there exists $\epsilon>0$ with 
\[
\alpha\cdot g(z+\epsilon\alpha\cdot Jg(x))<\alpha\cdot g(x)=0.
\]
 Therefore
\[
\alpha\cdot Jg(x)(z+\epsilon\alpha\cdot Jg(x)-x)\le 0.
\]
 We conclude from here and from
 \[
\alpha\cdot Jg(x)(z-x)=0\quad\textrm{that}\quad \alpha\cdot Jg(x)=0.
\]
On the other hand, let us choose in (\ref{10}) $\lambda=0$. Then we have
\[
\mu\cdot J g(x)\ne 0\quad\textrm{for all}\quad \mu\in M^*(x)
\]
 such that $\mu\ne 0$. Thus, we obtained a contradiction, which completes the proof. 
\end{proof}

Consider the scalar nonlinear programming problem

\bigskip
Minimize $f(x)$ subject to $g(x)\leqq 0$,\hfill (SP)
\bigskip

\noindent
where $f:X\to\R$ and $\g$ are given differentiable functions defined on some open set $X\subset\R^s$.

\begin{definition}[\cite{JOGO-1}] 
The problem {\rm (SP)} with  Fr\'echet differentiable data is called  Fritz John pseudoconvex (in short, FJ-pseudoconvex) iff  for all points $x\in X$  and $y\in X$ is satisfied the following implication:
\begin{equation}
\left.
\begin{array}{l}
f(y)< f(x) \\
g(x)\leqq 0,\; g(y)\leqq 0
\end{array}\right]
\quad\Rightarrow \quad
\left[
\begin{array}{l}
\nabla f(x)(y-x)<0\\
\nabla g(x)(y-x)<0,\; i\in I(x),
\end{array}
\right.
\end{equation}
where $I(x)$ is the set of active constraints.
\end{definition}



\section{FJ-pseudoinvex vector problems over cones}
\label{s4}
We proved in Section \ref{s3} that if the problem (P) is FJ-pseudoconvex, then every Fritz John stationary point is a weakly effective solution. In this section, we define the most general class of differentiable problems such that this property is satisfied.

Invex functions were introduced by Hanson \cite{han81}. Cone invex functions were introduced by Craven \cite{cra81}. It is proved by Ben-Israel and Mond \cite{ben86} that  a scalar function is invex if and only if every stationary point is a global minimum. Before Hanson invex functions have been studied abstractly without appropriate  definition in \cite{zha77}. KT-invex scalar problems with inequality constraints were introduced by Martin \cite{mar85}. In his paper, Martin also proved that a problem with inequality constraints is KT-invex if and only if every Kuhn-Tucker stationary point is a global minimizer.  The notion of KT-invexity were generalized later to multiobjective problems with inequality constraints (see Chapter 1 from the book \cite{ara2010} and \cite{osu98}). In this book and in the paper the reader can find generalizations of Martin's result to multiobjective problems.  In the present paper, we extend  Martin's results to the more general vector problem (P) with cone constraints, where the minimization is with respect to some cone. In this case, the proofs of the respective claims  are more complicated than the case when $C$ and $K$ are the Paretian cones.

\begin{definition} 
The problem {\rm (P)} with Fr\'echet differentiable data is called Fritz John pseudoinvex (in short, FJ-pseudoinvex) iff  for all points $x\in X$ and $y\in X$ there exists $\eta(x,y)\in\R^s$ such that the following implication holds:
\begin{equation}
\left.
\begin{array}{l}
f(y)\in f(x)-{\rm int}(C) \\
g(x)\in -K,\; g(y)\in -K
\end{array}\right]
\quad\Rightarrow \quad
\left[
\begin{array}{l}
J f(x)\eta(x,y)\in -{\rm int}(C^{**})\\
J g(x)\eta(x,y)\in -{\rm int}(M^{**}(x)).
\end{array}
\right.
\end{equation}
\end{definition}

\begin{theorem}\label{th5}
Let the problem {\rm (P)} be Fr\'echet differentiable. Suppose that $C$ and $K$ are closed convex cones, the polar cones $C^*$ and $K^*$ are pointed. Then {\rm (P)} is FJ-pseudoinvex if and only if every Fritz John vector critical point is a weak global minimizer.
\end{theorem}
\begin{proof}
Suppose that {\rm (P)} is FJ-pseudoinvex and $x$ is a Fritz John vector critical point, that is there exist $\lambda^0\in C^*$ and $\mu^0\in K^*$ which satisfy Conditions (\ref{1}). We prove that $x$ is a weakly efficient global solution. Assume the contrary that $x$ is not a weak global minimizer. Therefore, there exists a feasible point $y$ such that $f(y)\in f(x)-{\rm int}(C)$.
It follows from FJ-pseudoinvexity that there exists a vector $\eta\in\R^s$ with 
\[
J f(x)\eta \in -{\rm int}(C^{**})\quad\textrm{and}\quad
J g(x)\eta \in -{\rm int}(M^{**}(x)).
\]
 By Equations (\ref{1}) we have
\begin{equation}\label{3}
\lambda^0\cdot J f(x)\eta+\mu^0\cdot J g(x)\eta=0,\quad \mu^0 g(x)=0.
\end{equation}
On the other hand, by (\ref{3}) we obtain that $\mu^0\in M^*(x)$. Then, 
 by 
\[
\lambda^0\in C^*,\;\; J f(x)\eta\in -{\rm int}(C^{**})= -{\rm int}(C),\;\; \mu^0\in M^*(x)\;\;{\textrm and}\;\; J g(x)\eta\in -{\rm int}(M^{**}(x))
\]
 we have
\[
\lambda^0\cdot J f(x)\eta\le 0,\quad \mu^0\cdot J g(x)\eta\le 0.
\]
According to Lemma  \ref{lema7} at least one of these inequalities is strict, because $(\lambda^0,\mu^0)\ne (0,0)$. Thus we obtain a contradiction to Equation (\ref{3}).

We prove the converse claim. Suppose that each Fritz-John critical point is a weak global minimizer. We prove that (P) is FJ-pseudoinvex. Let
\[
x\in X, y\in X,\;\; g(x)\in -K,\;\; g(y)\in -K\quad\textrm{and}\quad f(y)\in f(x)-{\rm int}(C).
\]
 Therefore, $x$ is not a weak global minimizer. It follows from the hypothesis that $x$ is not a Fritz John critical point. Hence, there do not exist $(\lambda,\mu)\ne (0,0)$  such that Equations (\ref{1}) are satisfied by $x\in S$, $\lambda\in C^*$, $\mu\in K^*$.

Consider the set
\[
P:=\{p\in\R^s\mid p=\lambda\cdot Jf(x)+\mu\cdot Jg(x),\;\lambda\in C^*,\;\mu\in K^*,\; \mu\cdot g(x)=0\}.
\]
$P$ is a convex cone, whose vertex is  the origin $0$. We prove that $P$ is pointed. Indeed, let $p\in P\cap (-P)$. Then
\[
p=\lambda_1\cdot Jf(x)+\mu_1\cdot Jg(x),\quad  -p=\lambda_2\cdot Jf(x)+\mu_2\cdot Jg(x),
\]
where $\lambda_1$, $\lambda_2\in C^*$, $\mu_1$, $\mu_2\in M^*(x)$. It follows from here that
\[
p-p=0=(\lambda_1+\lambda_2)\cdot Jf(x)+(\mu_1+\mu_2)\cdot Jg(x).
\]
Since $x$ is not a Fritz John critical point, then we have $\lambda_1+\lambda_2=0$ and $\mu_1+\mu_2=0$. Therefore, $\lambda_2=-\lambda_1\in -C^*$, $\mu_2=-\mu_1\in -K^*$. Using that $C^*$ and $K^*$ are pointed we conclude that 
$\lambda_1=\mu_1=0$. Therefore $p=0$, which implies that $P$ is pointed.
Then it follows from Lemma \ref{lema6} that there exists a vector $\eta\in\R^s\setminus\{0\}$ such that $p\cdot\eta<0$ for all $p\in P$, $p\ne 0$. Since $p=0$ if and only if $(\lambda,\mu)=(0,0)$, then
\begin{equation}\label{4}
(\lambda\cdot Jf(x)+\mu\cdot Jg(x))\eta<0,\quad\forall (\lambda,\mu)\ne (0,0)\quad\textrm{such that}\quad \mu\cdot g(x)=0.
\end{equation}
Choose $\mu=0$ in (\ref{4}). It follows from (\ref{4}) and Lemma \ref{lema7} that
$J f(x)\eta\in -{\rm int}(C^{**})$. Choose $\lambda=0$ in (\ref{4}). Then it follows from (\ref{4})  that
\[
\mu\cdot Jg(x)\eta<0,\quad\forall \mu\in M^*(x), \mu\ne 0.
\]
 Then by the definition of $M^*(x)$ and Lemma \ref{lema7} we obtain that $J g(x)\eta\in -{\rm int}(M^{**}(x))$. Therefore, (P) is FJ-pseudoinvex.
\end{proof}

The following example shows that ${\rm int}(C)\ne\emptyset$ does not imply ${\rm int}(C^*)\ne\emptyset$:
\begin{example}
Consider the cone $C$ defined by $C=\{x=(x_1,x_2)\in\R^2\mid x_2\ge 0\}$. Then 
\[
C^*=\{x=(x_1,x_2)\in\R^2\mid x_1=0,\; x_2\ge 0\}.
\]
We have   ${\rm int}(C)\ne\emptyset$, but  ${\rm int}(C^*)=\emptyset$.
\end{example}

\end{document}